\documentclass[letterpaper,10pt]{article}
\usepackage{amsmath}
\usepackage{amssymb}
\usepackage{amsthm}
 \usepackage{theoremref}
\usepackage[british]{babel}
\usepackage{xcolor}
\usepackage{hyphenat}
\usepackage[colorlinks=true, allcolors=blue]{hyperref}
\usepackage{enumerate}

\theoremstyle{plain}
\newtheorem{lem}{Lemma}[section]
\newtheorem{thm}{Theorem}
\newtheorem*{thm*}{Theorem}
\newtheorem{cor}{Corollary}
\newtheorem*{cor*}{Corollary}

\theoremstyle{definition}

\newtheorem{df}{Definition}
\newtheorem{ob}{Observation}

\title{Quasi-strongly regular graphs on the flags of symmetric designs}
\author{Eugenia O'Reilly-Regueiro, Octavio B. Zapata-Fonseca}
\date{May 2025}

\begin{document}

\maketitle

\begin{abstract}
    This paper was inspired by a paper by Blokhuis and Brouwer \cite{BB} in which a definition of a graph on the flags of a biplane is given, and they prove that the graph corresponding to the unique $(11,5,2)$-biplane is determined by its spectrum. It is also inspired by the (\emph{different}) definition of flag-graph seen in the context of maps and abstract polytopes, \cite{Cu,Hu}. Here we use this definition for $(v,k,\lambda)$-BIBDs, and prove that if the design is symmetric then the graph is quasi-strongly regular. We will also use the definition given in~\cite{BB} for the case of biplanes and prove that this too, is a QSRG, (with different parameters). We investigate whether these graphs are determined by their spectra for some of the known biplanes.
\end{abstract}

\section{Introduction}
In a finite combinatorial incidence structure such as an abstract polytope which is simple and has a non-reflexive incidence relation, one can associate a ranked POSET, in which maximal chains are called {\em{flags}}. A balanced incomplete block design (BIBD) is a simple incidence structure and the incidence relation is non-reflexive, so the same applies, although this approach is not very useful in this case as flags have only two elements: a point and a block. In a ranked POSET, the {\em{flag-graph}} is defined to be the graph whose vertices are the flags of the POSET (maximal chains), and two vertices are adjacent if and only they differ in exactly one face (of a certain rank). See for example~\cite{Cu, Hu}. One of the motivations for this paper was to investigate this {\em{flag-graph}}, which we denote $\Gamma_1(D)$, in the case of a BIBD $D=(P,\mathcal{B})$. Some results we mention here have been published before, we cite and include them in the references, however for the sake of completeness we give our own proofs.\\

In this case, we prove the following:

\begin{thm*}[see Theorem \ref{t1} below]
     Let $D=(P,\mathcal{B})$ be a non-trivial $(v,b,r,k,\lambda)$-BIBD, and the flag-graph of $D$, $\Gamma_1(D)$ as defined above. Then one of the following holds:
     \begin{enumerate}[(i)]
         \item\label{t1i} $\lambda=1$ and $\Gamma_1(D)$ is a $(vr,k+r-2,\{r-2,k-2\},\{0,1\})$-AQSRG,
         \item $\lambda\geq 2$ and $\Gamma_1(D)$ is a $(vr,k+r-2,\{r-2,k-2\},\{0,1,2\})$-AQSRG.
     \end{enumerate}
 \end{thm*}

 \begin{cor*}[Corollary \ref{c1}]
 If $D=(P,\mathcal{B})$ is a $(v,k,\lambda)$-symmetric design, and $\Gamma_1(D)$ is the flag-graph of $D$ defined above, then one of the following holds:
 \begin{enumerate}[(i)]
     \item If $\lambda=1$ then $\Gamma_1(D)$ is a $(vk,2(k-1),k-2;0,1)$-QSRG,
     \item If $\lambda>1$ then $\Gamma_1(D)$ is a $(vk,2(k-1),k-2;0,1,2)$-QSRG.
 \end{enumerate}
 \end{cor*}

 \begin{thm*}[Theorem \ref{t3}]
  Let $D=(P,\mathcal{B})$ and $D'=(P',\mathcal{B}')$ be two $(v,k,\lambda)$-BIBDs, and $\Gamma_1(D)$, $\Gamma_1(D')$ their corresponding flag-graphs. Then $D\cong D'$ if and only $\Gamma_1(D)\cong\Gamma_1(D')$.  
\end{thm*}

The other motivation for this paper was the paper Spectral characterization of a graph on the flags of the eleven point biplane by Blockhuis and Brower~\cite{BB}. In this article the authors consider a biplane, which is a $(v,k,2)$-symmetric design, and define another graph whose vertices are the flags of the biplane. This definition does not lend itself well to other BIBDs. Then they prove that the graph so defined, in the case of the unique $(11,5,2)$ biplane is determined by its spectrum. We wanted to examine these graphs for other known biplanes.\\

For this case, we prove the following:\\

\begin{thm*}[Theorem \ref{t4}]
    Let $D=(P,\mathcal{B})$ be a $(v,k,2)$-biplane, and $\Gamma_2(D))$ the graph on the flags of $D$ as defined above. Then $\Gamma_2(D))$ is a $(vk,k-1,0;0,1,2)$-QSRG.
\end{thm*}

\begin{thm*}[Theorem \ref{thm:iso2}]
    Let $D=(P,\mathcal{B})$ and $D'=(P',\mathcal{B}')$ be two $(v,k,2)$-biplanes, and $\Gamma_2(D)$, $\Gamma_2(D')$ the corresponding graphs on their flags. Then $D\cong D'$ if and only if $\Gamma_2(D)\cong\Gamma_2(D')$.
\end{thm*}

In the first case, we denote the flag-graph of a non-trivial $(v,b,r,k,\lambda)$-BIBD $D=(P,\mathcal{B})$ as $\Gamma_1(D)$, and define its vertices to be the flags of $D$, with two vertices $(p,c)$ and $(q,d)$ adjacent if and only if $p=q$ or $c=d$. Note this is equivalent to the flags differing in exactly one face. With this definition, $\Gamma_1(D)$ is the line graph $L(\Gamma_D)$ of the incidence graph $\Gamma_D$ of $D$.\\

In 1965, Hoffman proved in~\cite{Hf} that if $D$ is a finite projective plane of order $n$, then the line graph $L(\Gamma_D)$ of its incidence graph $\Gamma_D$ is regular, and connected, with $(n+1)(n^2+n+1)$ vertices (the number of flags of a projective plane of order $n$), and the eigenvalues of its adjacency matrix $A(L(\Gamma_D))$ are $2n$, $-2$, and $n-2\pm\sqrt{n}$. Conversely, if $\Gamma$ is a regular and connected graph on $(n+1)(n^2+n+1)$ vertices such that its adjacency matrix $A(\Gamma)$ has such eigenvalues, then it is the line graph of the incidence graph of a projective plane of order $n$. That same year, in a similar paper~\cite{HR1} Hoffman and Ray-Chaudhury proved that if $D$ is a finite affine plane, then the line graph of its incidence graph is regular, connected, with $n^2(n+1)$ vertices, and the eigenvalues of its adjacency matrix are $2n-1$, $-2$, $\frac{1}{2}(2n-3)\pm\sqrt{(4n+1)}$, $n-2$. As above, they proved the converse: If $\Gamma$ is a regular and connected graph on $n^2(n+1)$ vertices and $A(\Gamma)$ has such eigenvalues, then $\Gamma$ is the line graph of the incidence graph of an affine plane of order $n$. Notice that in the case of the projective plane, which is a symmetric design, there are 4 eigenvalues, whereas in the case of the affine plane, which is not symmetric there are 5. In a third paper that year~\cite{HR2}, the same authors proved that if $D$ is a $(v,k,\lambda)$-symmetric design, then $L(\Gamma_D)$ is regular and connected on $vk$ vertices, and $A(L(\Gamma_D))$ has eigenvalues $2k-2$, $-2$, and $k-2\pm\sqrt{k-\lambda}$. As in the previous papers, they proved the converse: if $\Gamma$ is a regular and connected graph on $vk$ vertices with such eigenvalues, then it is the line graph of the incidence graph of a $(v,k,\lambda)$-symmetric design, however here there is exactly one exception. In the case $(v,k,\lambda)=(4,3,2)$ there is exactly one other graph with the same number of vertices and the same spectrum which is not isomorphic to the line graph of the incidence graph of a $(4,3,2)$-biplane, and the authors show this graph in the paper.\\

In 1969, in~\cite{RR}, Rao and Rao proved that if $D$ is an asymmetric $(v,b,r,k,1)$-BIBD with $r-2k+1<0$, then the line graph of the incidence graph of $D$ is almost as in our theorem~\ref{t1}~(\ref{t1i}), the only difference being that they proved that non-adjacent vertices have at most one common neighbour, we prove here that there are non-adjacent vertices with no common neighbours, and also with one common neighbour. Finally, in 1975, in~\cite{Db}, Doob proved that if $D$ is a $(v,b,r,k,1)$-BIBD and $r+k>18$, then the line graph of the incidence graph of $D$ is characterised by its spectrum. This is a summary of the results we found regarding the line graph of the incidence graph of a BIBD, which we denote $\Gamma_1(D)$.\\

The earliest reference we found for {\em{quasi-strongly regular graph}} (QSRG) is a paper by Antonucci~\cite{An}, from 1987.\\

In the second case, for a biplane $D$, we denote the graph on its flags by $\Gamma_2(D)$, and as defined in~\cite{BB}, its vertices are the flags of $D$, and two flags $(p,c)$ and $(q,d)$ are adjacent if and only if $\{p,q\}=c\cap d$. In \cite{BB} Blockhuis and Brouwer proved that is $D$ is the unique $(11,5,2)$-biplane and $\Gamma_2(D)$ is the graph defined on its flags as vertices and two vertices $(p,c)$ and $(q,d)$ adjacent if and only if $\{p,q\}=c\cap d$, then $\Gamma_2(D)$ is determined by its spectrum.\\

We would like to point out that we mean different things when we say a graph is {\em{characterised}} or {\em{determined}} by its spectrum. As we mentioned before, if $D$ is a $(v,k,\lambda)$-symmetric design (with $(v,k,\lambda)\neq(4,3,2)$), $\Gamma_D$ is its incidence matrix, and the line graph $L(\Gamma_D)=\Gamma_1(D)$ is its flag-graph, then as proved in \cite{HR2}, $\Gamma_1(D)$ is characterised by its spectrum. This means that if $H$ is a regular, connected graph on $vk$ vertices and it is cospectral with $\Gamma_1(D)$, then it is the line graph of the incidence graph of a $(v,k,\lambda)$-symmetric design $D'$, (that is, it is $\Gamma_1(D')$). This does not mean, however, that $H$ is isomorphic to $\Gamma_1(D)$ (unless $D$ and $D'$ are isomorphic). We do not have an equivalent result for $\Gamma_1(D)$ when $D$ is a non-symmetric BIBD, nor do we have an equivalent result for $\Gamma_2$ for biplanes. When Blockhuis and Brouwer proved in \cite{BB} that $\Gamma_2(D)$ when $D$ is the unique $(11,5,2)$-biplane is determined by its spectrum, it means that any graph with the  same spectrum is in fact isomorphic to $\Gamma_2(D)$.\\

For basic definitions and results on Design Theory, we refer the reader to~\cite{St}, and for Graph Theory, see~\cite{GR}.\\

The paper is organised as follows:\\
We start with Section \ref{prem} in which we mention some of the basic definitions and results in design and graph theory that we will be using. We continue with Section \ref{gc} in which we define $\Gamma_1$ and prove results relative to this graph. We would like to note that the fact that $\Gamma_1(D)$ is a QSRG when $D$ is symmetric is a corollary to the more general theorem we prove for the structure of $\Gamma_1$ for BIBDs in general. We follow with Section \ref{bp} where we shift our attention to biplanes, and prove our results for $\Gamma_2$, and finish with Section \ref{sp}, where we mention the spectra of some of these graphs.

\section{Preliminaries}\label{prem}
Given $v,b,r,k$, and $\lambda$ in $\mathbb{Z}^+$ with $k>1$, a $(v,b,r,k,\lambda)$-balanced incomplete block design (BIBD) $D=(P,\mathcal{B})$ is an incidence structure where $P$ is a set of $v$ points and $\mathcal{B}$ is a set of $b$ blocks, each of which is a $k$-subset of $P$, and every (non-ordered) pair of points is contained in exactly $\lambda$ blocks. We ask that $D$  be non-trivial, that is, $v-k>1$. In a $(v,b,r,k,\lambda)$-BIBD, each point is in exactly $r$ blocks. A $(v,v,k,k,\lambda)$-BIBD is called {\em{symmetric}}, and is usually denoted by $(v,k,\lambda)$-symmetric design. In a $(v,b,r,k,\lambda)$-BIBD, $vr=bk$ and $\lambda(v-1)=r(k-1)$. If the design is symmetric then $\lambda(v-1)=k(k-1)$. A {\em{flag}} in a $(v,b,r,k,\lambda)$-BIBD is an incident point-line (ordered) pair, so the number of flags is $vr=bk$. (See\cite{St}).\\

Given a $(v,b,r,k,\lambda)$-BIBD $D=(P,\mathcal{B})$, its {\em{incidence graph}} $\Gamma_D$ is a bipartite graph whose vertices are the disjoint sets $P\cup\mathcal{B}$, and two vertices (a point and a block) are adjacent if and only if they are incident in $D$. Given any graph $\Gamma$, its {line graph} $L(\Gamma)$ is a graph whose vertices are the edges of $\Gamma$ and two vertices are adjacent if and only if the corresponding edges in $\Gamma$ share exactly one common vertex (we do not want multiple edges). Given these definitions, if $D=(P,\mathcal{B})$ is a $(v,b,r,k,\lambda)$-BIBD, and $\Gamma_D$ is its incidence graph, then $L(D)$, the line graph of $D$ has the flags of $D$ as its vertices, and two vertices $(p,c)$, $(q,d)$ in $L(D)$ are adjacent if and only either $p=q$ or $c=d$. A graph $\Gamma$ is {\em{regular}} if every vertex has the same degree, it is $t$-regular if this number is $t$. A graph $\Gamma$ is a {\em{strongly regular graph}} (SRG) if it is regular, and also: every pair of adjacent vertices has the same number of common neighbours, and every pair of non-adjacent vertices has the same number of common neighbours. We would like to caution that the parameters generally used for SRGs overlap with those generally used for BIBDs, therefore we will change notation, and denote $\Gamma=(n,t,\eta,\mu)$ a strongly regular graph on $n$ vertices which is $t$-regular, such that any two adjacent vertices have $\eta$-common neighbours and any two non-adjacent neighbours have $\mu$-common neighbours. The earliest reference we found for a {\em{quasi-strongly regular graph}} is \cite{An}. A {\em{quasi-strongly regular graph}} (QSRG) $\Gamma$ is a regular graph in which any two common vertices have the same number of common neighbours, and any two non-adjacent vertices have a number of common neighbours which is in a finite set of possibly more than one element. That is, it differs from a SRG in that there can be more than one value for $\mu$. We will define an {\em{almost-quasi-strongly regular graph}} (AQSRG) $\Gamma$ to be a regular graph in which any two adjacent vertices have a number of common neighbours which is in a fixed finite set (as in a quasi-edge regular graph, that is, $\eta$ can take different values), and the same happens for non-adjacent vertices ($\mu$ can take different values).\\

\section{General case}\label{gc}
In this section we will adapt the flag-graph commonly used in the context of maps, abstract polytopes, and maniplexes, in which the vertices of the graph are the flags (maximal chains) and two flags are adjacent if and only if they differ in only one ($i-$) face. In the case of designs, a flag only has a point and a block, so differing in only one ``face'' is equivalent to being equal in only one ``face'', that is, having the same point, or the same block (but not both, we do not want loops). We will use a subscript 1 for this graph, to distinguish it from the other graph on flags we will define later, to which we will add a subscript 2.

\begin{df}
  Let $D=(P,\mathcal{B})$ be a non-trivial $(v,b,r,k,\lambda)$-BIBD. We define the \emph{flag-graph $\Gamma_1(D)=\big(V(\Gamma_1(D),E(\Gamma_1(D)\big)$} of $D$ as follows:
  \begin{enumerate}
      \item $V(\Gamma_1(D))=\{(p,c) |p\in P, c\in\mathcal{B}, p\in c\}$, that is, the flags of $D$.
      \item For any two different vertices $(p,c),(q,d)\in V(\Gamma_1(D))$, $\{(p,c),(q,d)\}\in E(\Gamma_1(D))$ if and only if either $p=q$ or $c=d$.
  \end{enumerate}
  When two vertices $(p,c)$ and $(q,d)$ form and edge, we say they are \emph{adjacent}, and we will denote this by $(p,c)\sim(q,d)$.
\end{df}

 Note $\Gamma_1(D)$ is \emph{simple}, that is, it has no loops nor multiple edges. 

 \begin{thm}\label{t1}
     Let $D=(P,\mathcal{B})$ be a non-trivial $(v,b,r,k,\lambda)$-BIBD, and the flag-graph of $D$, $\Gamma_1(D)$ as defined above. Then one of the following holds:
     \begin{enumerate}[(i)]
         \item $\lambda=1$ and $\Gamma_1(D)$ is a $(vr,k+r-2,\{r-2,k-2\},\{0,1\})$-AQSRG, 
         \item $\lambda\geq 2$ and $\Gamma_1(D)$ is a $(vr,k+r-2,\{r-2,k-2\},\{0,1,2\})$-AQSRG.
     \end{enumerate}
 \end{thm}

 \begin{cor}\label{c1}
 If $D=(P,\mathcal{B})$ is a $(v,k,\lambda)$-symmetric design, and $\Gamma_1(D)$ is the flag-graph of $D$ defined above, then one of the following holds:
 \begin{enumerate}[(i)]
     \item If $\lambda=1$ then $\Gamma_1(D)$ is a $(vk,2(k-1),k-2;0,1)$-QSRG,
     \item If $\lambda>1$ then $\Gamma_1(D)$ is a $(vk,2(k-1),k-2;0,1,2)$-QSRG.
 \end{enumerate}
 \end{cor}

 The only difference between the two cases is that when $\lambda=1$, the number of common neighbours of two non-adjacent vertices is either 0 or 1, and when $\lambda>1$, this number is either 0, 1, or 2. So we will do one proof for both cases and only distinguish them at the end.

 \begin{proof}
     Since the set of vertices of the graph is the set of flags of the design, one has only to count the flags of the design to obtain the number of vertices. Each point is in $r$ blocks, each of which has $k$ points, so that gives $vr$. (Alternatively one can use $bk$, since there are $b$ blocks and each block has $k$ points).\\

     Now we prove $\Gamma_1(D)$ is $(k+r-2)$-regular:\\
     
     Let $(p,c)\in V(\Gamma_1(D))$. Then $(p,c)$ is a flag of $D$, and $p\in c$. There are $k-1$ points $x_i\in c$ different from $p$ such that for each one, the (flag of $D$) vertex (of $\Gamma_1(D)$) $(x_i,c)$ is adjacent to $(p,c)$. This counts $k-1$ different neighbours of $(p,c)$.\\

     On the other hand, $p$ is in $r-1$ blocks $c_j$ all different from $c$, so there are $r-1$ different vertices $(p,c_j)$ adjacent to $(p,c)$. The neighbours in the previous paragraph are all different to the ones in this paragraph, since $x_i\neq p$ and $c_i\neq c$ for all $i\in\{1,\ldots,k-1\}$, $j\in\{1,\ldots,r-1\}$. Therefore $\Gamma_1(D)$ is $(k+r-2)$-regular.\\

     Next we prove that any two adjacent vertices have either $k-2$ or $r-2$ common neighbours.\\

     Suppose $(p,c)$ and $(q,d)$ are adjacent. Then either $p=q$ or $c=d$.\\
     If $p=q$ then $c\neq d$ and there are $r-2$ blocks $c_j\neq c,d$ such that for each $j\in\{1,\ldots,r-2\}$, $p\in c_j$. Therefore, for each $j$, $(p,c)\sim(p,c_j)\sim(p,d)$ which gives $r-2$ common neighbours for $(p,c)$ and $(q,d)$.\\
     Now suppose $c=d$. Then $p\neq q$ and there are $k-2$ points $x_i\in c$ different from $p,q$, for all $i\in\{1,\ldots,k-2\}$. This gives $k-2$ vertices $(x_i,c)$ such that $(p,c)\sim(x_i,c)\sim(q,c)$ for all $i\in\{1,\ldots,k-2\}$, so $(p,c)$ and $(q,d)$ have $k-2$ common neighbours.\\

     Finally, we check that if $\lambda=1$, then any two non-adjacent vertices have either 0 or 1 common neighbours, and if $\lambda>1$ then they have either 0, 1, or 2 common neighbours.\\

Let $(p,c)$ and $(q,d)$ be two non-adjacent vertices. Then $p\neq q$ and $c\neq d$. We have two possibilities in the case $\lambda=1$ and an additional possibility for $\lambda>1$. We leave this additional possibility at the end.
\begin{enumerate}
    \item $p\notin d$ and $q\notin c$. In this case there can be no common neighbour. Suppose, to the contrary, that $(z,c_0)$ is a common neighbour. Then $(p,c)\sim(z,c_0)\sim(q,d)$.\\
    If $z=p$ then $z\neq q$, so $c_0=d$ which implies $z=p\in d$, which is a contradiction.\\
    If $z\neq p$ then $c_0=c\neq d$ which implies $z=y$ so $y\in c$, another contradiction. This forces $(p,c)$, $(q,d)$ to have 0 common neighbours.\\
    Note we can always find two non adjacent vertices satisfying this condition. Take any two blocks $c$ and $d$. Since there are no repeated blocks, $c\neq d$, and they both have the same cardinality, $k<v-1$ since $D$ is non-trivial. This implies that there is a point $p\in c\setminus d$ and there is a point $q\in d\setminus c$.\\
    \item Without loss of generality  $p\in d$ and $q\notin c$. In this case they have one common neighbour: Since $p\in d$, there is a vertex $(p,d)$, and $(p,c)\sim(p,d)\sim(q,d)$. It remains to prove that this is the only common neighbour. Suppose $(z,c_0)$ is a common neighbour. Then $(p,c)\sim(z,c_0)\sim(q,d)$. If $z=p$ then $c_0\neq c$ which forces $c_0=d$ so $(z,c_0)=(p,d)$. If $z\neq p$ then $c_0=c\neq d$ which forces $z=q$ but this is a contradiction because $q\notin c$. Therefore $(p,c)$ and $(q,d)$ have only one common neighbour.\\
    We can also always find two vertices satisfying this condition. Take any flag $(p,c)$, since $p$ is in $r$ blocks, there is another block $d$ that contains $p$, and since there are no repeated blocks, $c\neq d$ so there is a point $q\in d\setminus c$.
    These are the only two possibilities if $\lambda=1$, however there is one more possibility for $\lambda>1$:
    \item $p\in d$ and $q\in c$. Notice this cannot happen if $\lambda=1$ because we have 2 points in 2 blocks. In this case, there are two common neighbours. As in the previous case, since $p\in d$ then $(p,d)$ is a vertex, and also, since $q\in c$ then $(q,c)$ is a vertex, they are different, and these are the only two common neighbours:\\
    Clearly $(p,c)\sim(p,d)\sim(q,d)$ and $(p,c)\sim(q,c)\sim(q,c)$. Now suppose $(z,c_0)$ is a common neighbour. Then $(p,c)\sim(z,c_0)\sim(q,d)$. If $z=p$ then $c\neq c_0=d$ so $(z,c_0)=(p,d)$, and if $z\neq p$ then $c_0=c$ and $z=y$ so $(z,c_0)=(y,c)$.\\
    This third possibility can also always occur provided $\lambda>1$, since any 2 points are in exactly $\lambda$ blocks. 
\end{enumerate}
     This proves the theorem.
 \end{proof}

 \begin{cor}
 If $D=(P,\mathcal{B})$ is a $(v,k,\lambda)$-symmetric design, and $\Gamma_1(D)$ is the flag-graph of $D$ defined above, then one of the following holds:
 \begin{enumerate}[(i)]
     \item If $\lambda=1$ then $\Gamma_1(D)$ is a $(vk,2(k-1),k-2;0,1)$-QSRG,
     \item If $\lambda>1$ then $\Gamma_1(D)$ is a $(vk,2(k-1),k-2;0,1,2)$-QSRG.
 \end{enumerate}
 \end{cor}
\begin{proof} 
    If $D$ is symmetric then $v=b$ and $k=r$, so the result follows.
\end{proof}

It now makes sense to ask whether having isomorphic designs is equivalent to having isomorphic graphs. We have the following:

\begin{thm}[\cite{Wh}]
\label{whit}
Let $\Gamma$ and $\Gamma'$ be connected graphs with isomorphic line graphs. Then $\Gamma$ and $\Gamma'$ are isomorphic unless one is $K_3$ and the other is $K_{1,3}$.
\end{thm}

\begin{ob}~\label{K3}
    If $D$ is a $(v,k,\lambda)$-BIBD, then its incidence graph $\Gamma_D$ is bipartite, and since $D$ has more than one point and one block, it follows that $\Gamma_D\neq K_3,K_{1,3}$.
\end{ob}

\begin{thm}\label{t3}
  Let $D=(P,\mathcal{B})$ and $D'=(P',\mathcal{B}')$ be two $(v,k,\lambda)$-BIBDs, and $\Gamma_1(D)$, $\Gamma_1(D')$ their corresponding flag-graphs. Then $D\cong D'$ if and only $\Gamma_1(D)\cong\Gamma_1(D')$.  
\end{thm}

\begin{proof}
    Let $D=(P,\mathcal{B})$ and $D'=(P',\mathcal{B}')$ be two $(v,k,\lambda)$-BIBDs, and $\Gamma_1(D)$, $\Gamma_1(D')$ their corresponding flag-graphs.\\
    
   First suppose $D\cong D'$, so there is a bijection $\alpha:P\to P'$ such that for any $p\in P$ and any $c\in\mathcal{B}$, $p\in c$ iff $\alpha(p)\in\alpha[c]$, where $\alpha[c]:=\{\alpha(x) | x\in c\}$. Since $\alpha$ induces a bijection on the sets of blocks $\mathcal{B}$, $\mathcal{B'}$ and preserves incidence and non-incidence, it induces a bijection $\overline{\alpha}$ on the set of flags of $D$ and $D'$, and so from $V(\Gamma_1(D))$ to $V(\Gamma_1(D'))$, where $\overline{\alpha}:V(\Gamma_1(D))\to V(\Gamma_1(D'))$ is defined by $\overline{\alpha}((p,c))=(\alpha(p),\alpha[c])$. It remains to show that it preserves adjacency and non-adjacency.\\
    Let $(p,c),(q,d)\in V(\Gamma_1(D))$. Then $(p,c)\sim(q,d)$ iff either $p=q$ or $c=d$, iff either $\alpha(p)=\alpha(q)$ or $\alpha[c]=\alpha[d]$, iff $(\alpha(p),\alpha[c])\sim(\alpha(q),\alpha[d])$ iff $\overline{\alpha}(p,c)\sim\overline{\alpha}(q,d)$.\\

For the converse, we will use Theorem~\ref{whit} and Observation~\ref{K3}. So now suppose $\Gamma_1(D)\cong\Gamma_1(D')$.\\

    Recall $\Gamma_D$ and $\Gamma_{D'}$ are the (directed) incidence graphs of $D$ and $D'$ respectively. Then $\Gamma_1(D)=L(\Gamma_D)$, the line graph of $\Gamma_D$, and similarly $\Gamma_1(D')=L(\Gamma_{D'})$. By Theorem~\ref{whit} and Observation~\ref{K3}, $\Gamma_D\cong\Gamma_{D'}$.\\

    Since $\Gamma_D\cong\Gamma_{D'}$, there is an bijection $\alpha:P\cup \mathcal{B}\to P'\cup \mathcal{B}'$ such that for any $p\in P$, $c\in \mathcal{B}$, $p\sim c$ if and only $\alpha(p)\sim \alpha(c)$. By definition of incidence graph, $p\sim c$ if and only if the point $p$ is incident with the block $c$, and $\alpha(p)\sim \alpha(c)$ if and only if $\alpha(p)$ and $\alpha(c)$ are incident. Both graphs $\Gamma_D$ and $\Gamma_{D'}$ are bipartite (with points on one set and blocks on the other) and directed (with edges starting at points and ending at blocks), therefore $\alpha\restriction _P:P\to P$, $\alpha\restriction _\mathcal{B}:\mathcal{B}\to\mathcal{B}$, and for any $p\in P$, $c\in \mathcal{B}$, $p\sim c$ if and only if $\alpha(p)$ is in $\alpha(c)$. That is, $\alpha$ induces a bijection of the points and of the of the blocks which preserves incidence, this implies $D\cong D'$.
\end{proof} 

\section{Biplanes}\label{bp}
 In this section we turn our attention to biplanes, which are $(v,k,2)$-symmetric designs. Given that in a $(v,k,\lambda)$-symmetric design the cardinality of the intersection of any two blocks is always $\lambda$, biplanes lend themselves to a nice definition of a graph on their set of flags, which was given by A. Blokhuis and A.E. Brouwer~\cite{BB} as follows:

\begin{df}
    Let $D=(P,\mathcal{B})$ be a $(v,k,2)$-biplane. We define a graph $\Gamma_2(D)=\big(V(\Gamma_2(D),E(\Gamma_2(D)\big)$ on the flags of $D$ as follows:
    \begin{enumerate}
        \item $V(\Gamma_2(D))=\{(p,c) |p\in P, c\in\mathcal{B}, p\in c\}$, that is, the flags of $D$.
        \item For any two vertices $(p,c),(q,d)\in V(\Gamma_2(D))$, $\{(p,c),(q,d)\}\in E(\Gamma_2(D))$ if and only if $c\cap d=\{p,q\}$.
    \end{enumerate}
\end{df}

\begin{thm}\label{t4}
    Let $D=(P,\mathcal{B})$ be a $(v,k,2)$-biplane, and $\Gamma_2(D))$ the graph on the flags of $D$ as defined above. Then $\Gamma_2(D))$ is a $(vk,k-1,0;0,1,2)$-QSRG.
\end{thm}

\begin{proof}
    As in the case of the flag-graph $\Gamma_1(D)$, the vertices are the flags of $D$, and there are $vk$ of them.\\

    We now prove that $\Gamma_2(D)$ is $(k-1)$-regular:\\
    Let $(p,c)\in V(\Gamma_2(D))$. Then $(p,c)$ is a flag of $D$, which means $p\in c$. There are $k-1$ points $x_i\in c\setminus\{p\}$ and $k-1$ blocks $c_i\neq c$ such that $c\cap c_i=\{p,x_i\}$ for all $i\in\{1,\ldots,k-1\}$. This implies there are $k-1$ vertices $(x_i,c_i)\in V(\Gamma_2(D))$ such that $(p,c)\sim(x_i,c_i)$, $i\in\{1,\ldots,k-1\}$, and hence $\Gamma_2(D))$ is $(k-1)$-regular.\\

    We now prove $\Gamma_2(D))$ is triangle-free:\\
    Suppose $(p,c),(q,d)\in V(\Gamma_2(D))$, and $(p,c)\sim(q,d)$. Then $c\cap d=\{p,q\}$. Now suppose $(z,c_0)\in V(\Gamma_2(D)))$ is such that $(p,c)\sim(z,c_0)\sim(q,d)$. Then $c\cap c_0=\{p,z\}$ and $c_0\cap d=\{z,d\}$. This implies $c_0\cap c\cap d=\{p,q,z\}$, a contradiction. Therefore adjacent vertices have no common neighbours.\\

    Finally we will prove that non-adjacent vertices can have 0, 1, or 2 common neighbours. We start by pointing out that if $(p,c)$ and $(q,d)$ are non-adjacent, then one of the following holds, (and there are always vertices that satisfy each case):
    \begin{enumerate}
        \item~\label{pq} $p=q$ and $c\neq d$,
        \item~\label{cd} $p\neq q$ and $c=d$, or
        \item~\label{nope} $p\neq q$ and $c\neq d$ and either:
        \begin{itemize}
            \item $p\in d$ and $q\notin c$,
            \item $p\notin d$ and $q\in c$,
            \item $p\notin d$ and $q\notin c$.
        \end{itemize}
    \end{enumerate}

    Case~\ref{pq}: If $p=q$ then the vertices are $(p,c)$ and $(p,d)$. Suppose $(z,c_0)$ is a common neighbour. Then $c\cap c_0=\{p,z\}=d\cap c_0$. Since $c\neq d$ and every pair of points is in exactly 2 blocks, this forces either $c_0=c$ or $c_0=d$, which implies $k=2$, a contradiction. Therefore in this case the two vertices have no common neighbours.\\

    Case~\ref{cd}: If $c=d$ then the vertices are $(p,c)$ and $(q,c)$. If $(z,c_0)$ is a common neighbour then $c\cap c_0=\{p,z\}=\{q,z\}=c\cap c_0$. This forces $p=q$, a contradiction, so in this case the two vertices again have no common neighbours.\\

    Case~\ref{nope}: For this case suppose that for some $t\in\mathbb{Z}^+$, there is a set of common neighbours $(z_i,c_i)$ of $(p,c)$ and $(q,d)$ for all $i\in\{1,\ldots, t\}$. Then $\forall i\in\{1,\ldots, t\}$, $c\cap c_i=\{z_i,p\}$ and $d\cap c_i=\{q,z_i\}$. Since $k>2$, $\forall i\in\{1,\ldots, t\}$ $c_i\neq c,d$, and the above intersections imply $ \{p,q\}\subseteq \displaystyle \bigcap_{i=1}^tc_i$ which forces $t\leq 2$, that is, $(p,c)$ and $(q,d)$ have at most 2 common neighbours.\\

    As mentioned above, there are always vertices for each of the three cases. This implies that there are always non-adjacent vertices with no common neighbours. We will finish the proof showing that we can always find non-adjacent vertices that have at least one common neighbour. For this, take two different blocks $c\neq d$, and suppose $c\cap d=\{x,y\}$. Take a point $p\in c$ such that $p\neq x,y$. There is a block $c_0$ such that $c_0\cap c=\{x,p\}$. Note that $c_0\neq c$ since $k>2$, and $c_0\neq d$ since $c\cap d=\{x,y\}$. There is a point $q\in d$ such that $c_0\cap d=\{x,q\}$. Note $q\neq y$, since $c_0\neq c$. First we claim $(p,c)\nsim (q,d)$. This is immediate since $c\cap d=\{x,y\}$ and $p\neq x,y$. Next we claim the vertex $(x,c_0)$ is a common neighbour of $(p,c)$ and $(q,d)$. This is also immediate, since by construction $c_0\cap c=\{p,x\}$ which implies $(p,c)\sim(x,c_0)$ and $c_0\cap d=\{q,x\}$ which implies $(x,c_0)\sim(q,d)$. 
\end{proof}

\begin{ob}
    We would like to point out that although non-adjacent vertices can have 0, 1, or 2 common neighbours, we do not claim that the values 1 and 2 both occur in every case. In the examples we examined there were either 0 and 1, or 0 and 2 common neighbours in each one.
\end{ob}

\begin{lem}~\label{beta}
    Let $D=(P,\mathcal{B})$ and $D'=(P',\mathcal{B}')$ be two non-trivial $(v,k,2)$-biplanes, and $\Gamma_2(D)$ and $\Gamma_2(D')$ the corresponding graphs on their flags, and suppose $\Gamma_2(D)\cong\Gamma_2(D')$. If $\beta:V(\Gamma_2(D))\to V(\Gamma_2(D'))$ is such an isomorphism, then $\beta$ induces a bijection $\beta_P:P\to P'$ and a bijection $\beta_{\mathcal{B}}:\mathcal{B}\to\mathcal{B}'$ such that
     for any $p_1,p_i,p_j\in P$ and $c_1,c_i,c_j\in\mathcal{B}$ such that $(p_1,c_i),(p_1,c_j),(p_i,c_1),(p_j,c_1)\in V(\Gamma_2(D))$, the following hold:
    \begin{enumerate}[(i)]
        \item~\label{b1} If $\beta(p_1,c_i)=(p_1',c_i')$ and $\beta(p_1,c_j)=(q',c_j')$, then $p_1'=q'$, and
        \item~\label{b2} If $\beta(p_i,c_1)=(p_i',c_1')$ and $\beta(p_j,c_1)=(p_j',d')$, then $c_1'=d'$,
    \end{enumerate}
    that is, for any $p\in P,c\in\mathcal{B}$, if $p\in c$ then $\beta(p,c)=(\beta_P(p),\beta_\mathcal{B}(c))$.
\end{lem}

\begin{proof}
    Let $D=(P,\mathcal{B})$ and $D'=(P',\mathcal{B}')$ be two non-trivial $(v,k,2)$-biplanes, and $\Gamma_2(D)$, $\Gamma_2(D')$ the corresponding graphs on their flags, and suppose there is an isomorphism $\beta:V(\Gamma_2(D))\to V(\Gamma_2(D'))$, such that for any $(p,c),(q,d)\in V(\Gamma_2(D))$, $(p,c)\sim(q,d)\iff\beta(p,c)\sim\beta(q,d)$.\\
    
    For~(\ref{b1}), let $c_1\in\mathcal{B}$ with $p_1,p_i,p_j\in c_1$ such that $\{p_1,p_i\}=c_1\cap c_i$ and $\{p_1,p_j\}=c_1\cap c_j$ for some unique $c_i,c_j\in\mathcal{B}$, and such that $\beta(p_1,c_i)=(p_1',c_i')$ and $\beta(p_1,c_j)=(q',c_j')$. Then there is a point $x_{ij}\in P$ such that $c_i\cap c_j=\{p_1,x_{ij}\}$. This point $x_{ij}$ might or might not be in $c_1$, however we do have the following adjacencies in $\Gamma_2(D)$:\\

$(p_1,c_i)\sim(x_{ij},c_j)\iff\beta(p_1,c_i)\sim\beta(x_{ij},c_j)\iff(p_1',c_i')\sim\beta(x_{ij},c_j).$\\
    
    Also:\\
    
$(p_1,c_j)\sim(x_{ij},c_i)\iff\beta(p_1,c_j)\sim\beta(x_{ij},c_i)\iff(q',c_j')\sim\beta(x_{ij},c_i)$.\\

    If $p_1\neq q'$, then there are two blocks $d_1',d_2'\in\mathcal{B}'$ such that $d_1'\cap d_2'=\{p_1',q'\}$, and there are two adjacent pairs of vertices in $\Gamma_2(D')$: $(p_1',d_1')\sim(q',d_2')$ and $(p_1',d_2')\sim(q',d_1')$.\\

    Since $\beta$ is an isomorphism, these vertices must correspond to two pairs of adjacent vertices in $\Gamma_2(D)$ given by the elements of $(p_1,c_i)$ and $(p_1,c_j)$. Since $p_1$ is the same point of $D$ in both vertices, they must then correspond to the intersection of $c_i\cap c_j=\{p_1,x_{ij}\}$. That is:
    \begin{align*}
    &\{\beta^{-1}(p_1',d_1'),\beta^{-1}(p_1',d_2'),\beta^{-1}(q',d_1'),\beta^{-1}(q',d_2')\}\\
    &\qquad\qquad=\{(p_1,c_i),(p_1,c_j),(x_{ij},c_i),(x_{ij},c_j)\}.
\end{align*}
    Without loss of generality, suppose $c_i'=d_1'$ and $c_j'=d_2'$. We have the following:\\

$(p_1,c_i)\sim(x_{ij},c_j)\iff\beta(p_1,c_i)\sim\beta(x_{ij},c_j)\iff(p_1',c_i')\sim\beta(x_{ij},c_j)\iff(p_1',d_1')\sim(q',d_2')$, so $\beta(x_{ij},c_j)=(q',d_2')=(q',c_j')=\beta(p_1,c_j)$, which is a contradiction.\\

The proof of~(\ref{b2}) is very similar. Now suppose there is a block $c_1\in\mathcal{B}$ with $p_i,p_j\in c_1$ such that $\beta(p_i,c_1)=(p_i',c_1')$ and $\beta(p_j,c_1)=(p_j',d')$ with $c_1'\neq d'$.\\

Since $k>3$, as before, let $p_1\in c_1$ be another point, and $c_i,c_j\in\mathcal{B}$ be the blocks such that $\{p_1,p_i\}=c_1\cap c_i$ and $\{p_1,p_j\}=c_1\cap c_j$.\\

Again, we have the following adjacencies in $\Gamma_2(D)$:\\

$(p_1,c_i)\sim(p_i,c_1),$ and $(p_1,c_j)\sim(p_j,c_1)$.\\

This happens if and only if:\\

$\beta(p_1,c_i)\sim\beta(p_i,c_1)=(p_i',c_1'),$ and $\beta(p_1,c_j)\sim\beta(p_j,c_1)=(p_j',d')$.\\

Since $c_1'\neq d'$, there are two points $q_1',q_2'\in P'$ such that $c_1'\cap d'=\{q_1',q_2'\}$. This induces two pairs of adjacent vertices in $\Gamma_2(D')$:\\

$(q_1',c_1')\sim(q_2',d')$, and $(q_1',d')\sim(q_2',c_1')$.\\

Since $\beta$ is an isomorphism, these must correspond to two pairs of adjacent vertices in $\Gamma_2(D)$ induced by the elements of the flags $(p_i,c_1)$ and $(p_j,c_1)$. Since the block $c_1$ is the same in both vertices, they must correspond to the intersection $\{p_i,p_j\}=c_1\cap d$ for some block $d\in\mathcal{B}$. Therefore:\\

$\{\beta(p_1,c_i),\beta(p_1,c_j),\beta(p_i,c_1),\beta(p_j,c_1)\}=\{(q_1',c_1'),(q_2',c_1'),(q_1',d'),(q_2',d')\}$.\\

Suppose without loss of generality that $p_i'=q_1'$. Then:\\

$\beta(p_1,c_i)\sim\beta(p_i,c_1)=(p_i',c_1')=(q_1',c_1')\sim(q_2',d')$, so $\beta(p_1,c_i)=(q_2',d')$.\\

From above:\\

$\beta(p_1,c_j)\sim\beta(p_j,c_1)=(p_j',d')=(q_2',d')$, a contradiction which completes the proof.
\end{proof}

\begin{thm}
\label{thm:iso2}
    Let $D=(P,\mathcal{B})$ and $D'=(P',\mathcal{B}')$ be two $(v,k,2)$-biplanes, and $\Gamma_2(D)$, $\Gamma_2(D')$ the corresponding graphs on their flags. Then $D\cong D'$ if and only if $\Gamma_2(D)\cong\Gamma_2(D')$.
\end{thm}

\begin{proof}
    Let $D=(P,\mathcal{B})$ and $D'=(P',\mathcal{B}')$ be two $(v,k,2)$-biplanes, and $\Gamma_2(D)$, $\Gamma_2(D')$ the corresponding graphs on their flags.\\
    
    First suppose $D\cong D'$. Then there is a bijection $\beta:P\to P'$ such that for any $p\in P$ and any $c\in\mathcal{B}$, $p\in c$ if and only $\beta(p)\in\beta[c]$, where $\beta[c]=\{\beta(p_1), \ldots \beta(p_k)|p_1,\ldots, p_k\in c\}$.\\

    Now consider $\Gamma_2(D)$ and $\Gamma_2(D')$, and let $(p,c),(q,d)\in V(\Gamma_2(D))$. Then $p,q\in P$, $c,d\in\mathcal{B}$, $p\in c$ and $q\in d$, so $\beta(p),\beta(q)\in P'$, $\beta[c],\beta[d]\in\mathcal{B}'$, $\beta(p)\in \beta[c]$ and $\beta(q)\in \beta[d]$, which implies $(\beta(p),\beta[c]),(\beta(q),\beta[d])\in V(\Gamma_2(D'))$.\\
    
    By definition, $(p,c)\sim (q,d)$ if and only if $c\cap d=\{p,q\}$, and $(\beta(p),\beta[c])\sim(\beta(q),\beta[d])$ if and only if $\beta[c]\cap\beta[d]=\{\beta(p),\beta(q)\}$. Notice, however, that since $\beta$ is an isomorphism between $D$ and $D'$ it follows that $c\cap d=\{p,q\}$ if and only if $\beta[c]\cap\beta[d]=\{\beta(p),\beta(q)\}$, so $\Gamma_2(D)\cong\Gamma_2(D')$.\\

    Now suppose $\Gamma_2(D)\cong\Gamma_2(D')$, with an isomorphism $\beta:V(\Gamma_2(D))\to V(\Gamma_2(D'))$ such that for any $(p,c),(q,d)\in V(\Gamma_2(D))$, $(p,c)\sim(q,d)\iff\beta(c,d)\sim\beta(q,d)$. By lemma~\ref{beta}, $\beta$ induces a bijection $\beta_P:P\to P'$ and a bijection $\beta_\mathcal{B}:\mathcal{B}\to\mathcal{B}'$, such that $\forall p\in P, c\in\mathcal{B}$, with $(p,c)\in V(\Gamma_2(D)$ there are $p'\in P', c'\in\mathcal{B}'$ such that $\beta(p,c)=(p',c')=(\beta_P(p),\beta_\mathcal{B}(c))$.\\

    Let $p\in P$ and $c\in\mathcal{B}$. Then:\\

    $p\in c\iff(p,c)\in V(\Gamma_2(D))\iff\beta(p,c)\in V(\Gamma_2(D'))\iff(\beta_P(p),\beta_\mathcal{B}(c))\in V(\Gamma_2(D'))\iff\beta_P(p)\in\beta_\mathcal{B}(c)$, so $D\cong D'$.
\end{proof}

\section{Spectra}\label{sp}
As we mentioned before, our original motivation comes from the work of Blokhuis and Brouwer~\cite{BB}, in which they proved that the graph $\Gamma_2(D)$ on the flags of the unique $(11,5,2)$-biplane $D$ is determined by the spectrum of its adjacency matrix. 
In this section we describe the spectra of the graphs $\Gamma_1(D)$ and $\Gamma_2(D)$ for all the biplanes with at most 16 points.

\subsection{The spectrum of the graph $\Gamma_1(D)$}

Recall that the incidence graph $\Gamma_D$ of a $(v,b,r,k,\lambda)$-BIBD $D = (P,\mathcal{B})$ is the bipartite  graph with vertex set $P\cup\mathcal{B}$, where two vertices are adjacent if they form an incident point-block pair, that is,  if they form a flag of $D$. \\

Any two non-isomorphic BIBDs with the same parameters give rise to a pair of non-isomorphic incidence graphs that have the same spectrum. 
For example, the four non-isomorphic $(6,20,10,3,4)$-BIBDs give rise to four non-isomorphic cospectral  graphs  with spectrum
\[
\sqrt{30}, (\sqrt{6})^5, 0^{14}, (-\sqrt{6})^5, -\sqrt{30},
\]
where we write $\theta^m$ to denote that the eigenvalue  $\theta$ has multiplicity $m$. 
\\

A bipartite graph such that all the vertices in the same part of the bipartition have the same degree is called \emph{biregular}. 
Every regular bipartite graph is biregular. 
The incidence graph $\Gamma_D$ of a $(v,b,r,k,\lambda)$-BIBD $D$ is biregular with degrees $r$ and $k$.
If $N$ is the point-block incidence matrix of  $D$, the adjacency matrix $A$ of $\Gamma_D$ has the form
\[
A = 
\begin{bmatrix}
0 & N\\
N^\top & 0
\end{bmatrix}.
\]
Since 
\[
A^2 = 
\begin{bmatrix}
NN^\top & 0\\
0 & N^\top N
\end{bmatrix},
\]
the characteristic polynomial of $A^2$ factors as the product of the characteristic polynomials of $NN^\top$ and  $N^\top N$. 
But $NN^\top$ and  $N^\top N$ have the same non-zero eigenvalues, so if $J$ denotes the all-ones matrix, then $NN^\top =(r-\lambda)I + \lambda J$, and the spectrum of $NN^\top$ and $N^\top N$ is 
\[
rk, (r-\lambda)^{v-1}.
\]
Because the non-zero eigenvalues of $A$ are  the square roots of the non-zero eigenvalues of $A^2$, the spectrum of the incidence graph $\Gamma_D$ is  
\[
\sqrt{rk}, (\sqrt{r-\lambda})^{v-1}, 0^{b-v}, (-\sqrt{r-\lambda})^{v-1}, -\sqrt{rk}.
\]\\

The flag-graph $\Gamma_1(D)$ of a $(v,b,r,k,\lambda)$-BIBD $D$ is, by definition, the line graph $L(\Gamma_D)$ of the incidence graph $\Gamma_D$. 
Since $\Gamma_D$ is a biregular graph, it follows that the spectrum of  $\Gamma_1(D)$  is
\begin{align*}
r+k-2,&\left(\frac{r+k-4+\sqrt{(k-r)^2-4(r-\lambda)} }{2}\right)^{v-1}, (k-2)^{b-v}, \\
&\left(\frac{k+r-4 - \sqrt{(k-r)^2-4(r-\lambda)}}{2}\right)^{v-1}, (-2)^{bk-b-v+1}.
\end{align*}
In the example above, the flag-graph $\Gamma_1(D)$ of any of the four non-isomorphic $(6, 20, 10, 3, 4)$-BIBDs $D$ has spectrum
\[
11, \left(\frac{9+\sqrt{73}}{2}\right)^{5}, 1^{14} ,\left(\frac{9-\sqrt{73}}{2}\right)^{5}, (-2)^{35}.
\]\\
 
Let $D$ be a $(v,k,\lambda)$-symmetric design.  
By definition, the incidence graph $\Gamma_D$  is regular bipartite with spectrum \[k, (\sqrt{k-\lambda})^{v-1}, (-\sqrt{k-\lambda})^{v-1} , -k.\] 
Since the flag-graph $\Gamma_1(D)$ is, by definition, the line graph $L(\Gamma_D)$,  the  spectrum of $\Gamma_1(D)$ is 
\[2k-2, (k - 2 + \sqrt{k-\lambda})^{v-1}, (k - 2-\sqrt{k-\lambda} )^{v-1}, (-2)^{vk - 2v + 1}.\] 
As mentioned earlier, in \cite{HR2} Hoffman and Ray-Chaudhuri proved that if a graph is regular and connected on $vk$ vertices with these eigenvalues then it is $\Gamma_1(D)$ for some $(v,k,\lambda)$-symmetric design $D$ (with exactly one exception if $(v,k,\lambda)=(4,3,2)$).
If $D$ is a $(v,k,2)$-biplane,  the spectrum of $\Gamma_1(D)$ is 
\[
2(a+1), (a + \sqrt{a})^{v-1}, (a -\sqrt{a} )^{v-1}, (-2)^{va + 1}
\]
where $a = k-2$.\\

Now we describe the spectra of the flag-graphs $\Gamma_1(D)$ when $D$ is a $(v,k,2)$-biplane with $v  \leq 16$ points.\\ 

If $D$ is the unique $(4,3,2)$-biplane, then $a= 3-2 = 1$ and hence $\Gamma_1(D)$ 
has spectrum $4, 2^3, 0^3, (-2)^5$. This graph is not determined by its spectrum, since there is a unique graph $\Gamma$ with the same spectrum as $\Gamma_1(D)$ which is not a line graph. 
This graph $\Gamma \not\cong\Gamma_1(D)$ is described in~\cite{HR2}.\\ 

If $D$ is the unique $(7,4,2)$-biplane, 
then $a = 4-2 =2$ and hence the spectrum of $\Gamma_1(D)$ is $6, (2+\sqrt{2})^6, (2-\sqrt{2})^6,(-2)^{15}$. 
As a consequence of the uniqueness of $D$, we have that $\Gamma_1(D)$ is the unique graph up to isomorphism with this spectrum.
Therefore, $\Gamma_1(D)$ is determined by its spectrum.\\ 

If $D$ is the unique $(11,5,2)$-biplane, $a=5-2=3$, and the spectrum of $\Gamma_1(D)$ is $8,(3+\sqrt{3})^{10},(3-\sqrt{3})^{10},(-2)^{34}$. 
The uniqueness of $D$ implies that $\Gamma_1(D)$ is determined by its spectrum.\\

Let $D_1,D_2$ and $D_3$ be the three non-isomorphic $(16,6,2)$-biplanes. 
Since these three biplanes have the same parameters, the  graphs $\Gamma_1(D_1)$, $\Gamma_1(D_2)$ and $\Gamma_1(D_3)$ are cospectral with  spectrum $10,6^{15},2^{15},(-2)^{65}$. 
Since these three biplanes are non-isomorphic, by Theorem~\ref{thm:iso2} we have that  $\Gamma_1(D_1)$, $\Gamma_1(D_2)$ and $\Gamma_1(D_3)$ are pairwise non-isomorphic. 
It follows that these graphs  are not determined by their spectrum.

\subsection{The spectrum of the graph $\Gamma_2(D)$}
The computations in this section were carried out in sage~\cite{sage}. 
The graph $\Gamma_2(D)$ defined on the flags of the unique $(4,3,2)$-biplane $D$ is the disjoint union of three $4$-cycles $C_4$. Therefore, the graph $\Gamma_2(D)$ is determined by its spectrum.\\

The graph $\Gamma_2(D)$ on the flags of the unique $(7,4,2)$-biplane $D$ is the Coxeter graph (see~\cite[Section 4.6]{GR}). 
This graph is  determined by its spectrum.\\

Blokhuis and Brouwer \cite{BB} proved that the graph $\Gamma_2(D)$ of the unique $(11,5,2)$-biplane $D$ is determined by its spectrum.\\

Let $D_1$, $D_2$ and $D_3$ be the three non-isomorphic $(16,6,2)$-biplanes. 
\begin{enumerate}
\item The graph $\Gamma_2(D_1)$ is the disjoint union of six graphs. 
Each of these six graphs is isomorphic to the Clebsch graph, which is determined by its spectrum (see~\cite[Section 10.6 ]{GR}).
Therefore, the graph of $\Gamma_2(D_1)$ is  determined by its spectrum.

 \item The graph $\Gamma_2(D_2)$ is the disjoint union of three graphs. 
Each of these three graphs is isomorphic to a graph $\Gamma$. 
The spectrum of $\Gamma$ is $5^1, 1^{18}, (1+2\sqrt{2})^2, (1-2\sqrt{2})^2, (-3)^9$. 
We do not know whether this graph $\Gamma$ is determined by its spectrum. 
 
\item The graph $\Gamma_2(D_3)$ is the disjoint union of two non-isomorphic graphs. The spectrum of one of them is
$5^1, 1^{34}, (1+2\sqrt{2})^6,(1-2\sqrt{2})^6, (-3)^{17}$ and the spectrum of the other is $5^1,3^{4},1^{14},(-1)^{4},(-3)^9$. 
We do not know if any of these graphs is determined by its spectrum.
\end{enumerate}

\end{document}